\documentclass[12pt]{article}
\usepackage{amsmath}
\usepackage{amsthm}
\usepackage{subfigure}
\usepackage[xetex]{graphicx}
\usepackage{epstopdf}
\usepackage{mathrsfs}
\usepackage{amsmath}
\usepackage{amssymb}
\usepackage{epsfig}
\usepackage{amsmath,amssymb,amsfonts,amsthm,fancyhdr}
\usepackage{epsfig,graphicx,picins,picinpar,subfigure}
\usepackage{pstricks}
\usepackage{cases}
\usepackage{xcolor}
\usepackage{colortbl}
\definecolor{mygray}{gray}{.9}
\usepackage{stmaryrd}

\usepackage[numbers,sort&compress]{natbib}
\newtheorem{theorem}{Theorem}[section]

\newtheorem{remark}[theorem]{Remark}
\theoremstyle{nonumberplain}

\usepackage{layout}
\usepackage{lastpage}
\usepackage{fancyhdr}
\numberwithin{equation}{section} 
\begin{document}
\title{ \large \textbf{ Last-iterate convergence of modified 
predictive method via high-resolution differential equation on bilinear game}}
\author{\footnotesize     Keke Li$^{a,b}$, 
	Xinmin 
Yang$^{b}$\thanks{Corresponding author.\textit{E-mail addresses}:  xmyang@cqnu.edu.cn}\\
\textit{$^{a}$}{ \footnotesize  School of Mathematical Sciences, University of Electronic Science and Technology of China, 
}\\
\textit{}{ \footnotesize   
	Chengdu 611731, China
}\\
\textit{$^{b}$}{ \footnotesize School of Mathematics Science, 
Chongqing Normal University, Chongqing 400047, China
}}
\date{}

\maketitle
 {
 {\textbf{Abstract:} 
 	This paper discusses the convergence of the modified predictive method (MPM) proposed by Liang and stokes corresponding to high-resolution differential equations (HRDE) in bilinear games. First, we present the high-resolution differential equations (MPM-HRDE) corresponding to the MPM. Then, we discuss the uniqueness of the solution for MPM-HRDE in bilinear games. Finally, we provide the convergence results of MPM-HRDE in bilinear games. The results obtained in this paper address the gap in the existing literature and extend the conclusions of related works.
}
}
\section{Introduction}
\indent
\par

Gradient-based optimization algorithms can be understood from the perspective of ordinary differential equations (ODEs). Significant progress has been made in using ODEs to understand discrete-time algorithms (DTAs), please see references \cite{jordan2018dynamical, wibisono2016variational, wilson2021lyapunov,  su2016differential,  shi2018understanding, shi2022understanding, shi2019acceleration,  yuan2023analysis}. Understanding DTA from the perspective of ODEs has two main advantages: the convergence analysis of ODEs is usually more straightforward than that of DTAs, and the advanced analytical tools from the ODE literature can help provide more fundamental insights into the behavior of DTAs \cite{lu2022sr}.
However, relative to the significant role played by continuous-time analysis in single-objective optimization, the advantages of continuous-time analysis have not been fully utilized in min-max optimization problems \cite{chavdarova2021last, chavdarova2022continuous, chavdarova2023last}, such as those represented by generative adversarial networks (GAN) \cite{goodfellow2014generative}.

In fact, some first-order saddle point discrete optimization algorithms, such as the gradient descent ascent algorithm (GDA), the extragradient method (EG), and the optimistic gradient descent ascent algorithm (OGDA) mentioned in references \cite{chavdarova2021last, chavdarova2023last}, exhibit different convergence behaviors in simple bilinear games (a specific instance of GANs). However, in the limit as the step size tends to zero, they all lead to the same ODE. Therefore, a general ODE framework cannot distinguish between these different discrete algorithms.
To leverage the advantages of ODEs in the analysis of discretization algorithms in min-max problems, establishing different correspondences between DTA and their respective ODE convergence behaviors can assist in analyzing the convergence of DTAs using ODEs. Based on research results on ODEs in numerical optimization, Lu \cite{lu2022sr} proposed an $O(s^r)$-resolution ODE framework to analyze the behavior of general DTAs.
Inspired by the research ideas presented in \cite{lu2022sr}, to enhance the distinguishability between ODEs corresponding to different DTAs in min-max problems, Chavdarova et al. \cite{chavdarova2021last, chavdarova2023last} introduced high-resolution differential equations (HRDE) to design differential equation models for min-max problems. They established the convergence of GDA, EG, OGDA, and lookahead  methods (LA) corresponding to HRDE on bilinear games.

Note that the literature \cite{chavdarova2021last, chavdarovalast2021OPT, chavdarova2023last} discusses the convergence of the high-resolution differential equation  corresponding to the extragradient method in bilinear  games. However, the convergence of the high-resolution differential equation  corresponding to the modified predictive method has not been discussed. Therefore, one of the main motivations of this paper is to discuss the convergence of the high-resolution differential equations corresponding to the the modified predictive method in bilinear games.
On the other hand, it is noted that references \cite{chavdarova2021last, chavdarova2023last} have shortcomings in proving the uniqueness of the MPM-HRDE solution in bilinear games and impose overly strong constraints on matrix $\boldsymbol{A}$ when proving convergence. Therefore, addressing these shortcomings in the proof process and relaxing these conditions is another primary motivation of this paper. 

\section{Preliminaries}
\indent
\par

Unless otherwise stated, the notation used in this paper has the same meaning as in references \cite{chavdarova2021last, chavdarova2023last}. Lowercase bold letters represent vectors, 
$\mathfrak{R}(c)$ and 
$\mathfrak{I}(c)$ denote the real and imaginary parts of a complex number 
$c=\mathfrak{R}(c)+ \mathfrak{I}(c) \cdot i \in \mathbb{C}$, respectively.  
$\| \boldsymbol{x} \|_2$ and 
$\| \boldsymbol{A} \|_{\rm F}$ represent the Euclidean norm of a vector and the Frobenius norm of a matrix $\boldsymbol{A}$, respectively.

\subsection{Modified predictive method}
\indent
\par
In this subsection, we introduce several notations that will be used later and the iterative format of the modified predictive algorithm primarily discussed in this paper.
Denote $\boldsymbol{z}  \triangleq (\boldsymbol{x}, \boldsymbol{y}) \in \mathbb{R}^{d}$, where $d = d_1 + d_2 $, and define the vector field $V: \mathbb{R}^{d} \rightarrow \mathbb{R}^{d}$ and its Jacobian $J$ as follows:
\begin{equation*}
	V(\boldsymbol{z})=\left[\begin{array}{c}
		\nabla_x f(\boldsymbol{z}) \\
		-\nabla_y f(\boldsymbol{z})
	\end{array}\right], \quad J(\boldsymbol{z})=\left[\begin{array}{cc}
		\nabla_x^2 f(\boldsymbol{z}) & \nabla_y \nabla_x f(\boldsymbol{z}) \\
		-\nabla_x \nabla_y f(\boldsymbol{z}) & -\nabla_y^2 f(\boldsymbol{z})
	\end{array}\right] .
\end{equation*}

The modified predictive method (MPM) mentioned in references \cite{liang2019interaction, li2024, li2024Alleviating} is a variant of the EG. The idea is to obtain gradient information at an extrapolated point $\boldsymbol{z}_{t+\frac{1}{2}}$ using a prediction step $\boldsymbol{z}_{n+\frac{1}{2}}=\boldsymbol{z}_n-\alpha 
V\left(\boldsymbol{z}_n\right)$, and then update at the current point $\boldsymbol{z}_{n}$. The specific iterative format is as follows:
 
\begin{equation}
	\boldsymbol{z}_{n+\frac{1}{2}}=\boldsymbol{z}_n-\alpha 
	V\left(\boldsymbol{z}_n\right), \quad
	\boldsymbol{z}_{n+1}=\boldsymbol{z}_n-\gamma
	V(\boldsymbol{z}_{n+\frac{1}{2}} ). \tag{MPM}
\end{equation}

\subsection{High-resolution differential equations of modified predictive method}
\indent
\par

In this subsection, we turn to deriving the $\mathcal{O}(\alpha)$-ODEs for modified predictive method. 
The specific derivation process of MPM corresponding to MPM-HRDE can be referenced from the construction process of EG corresponding to EG-HRDE in the literature \cite{chavdarova2021last,  chavdarova2023last}. By combining the iterative format of MPM with equation (3) from the references \cite{chavdarova2021last, chavdarova2023last}, we can obtain:
\begin{equation*}
	(\dot{\boldsymbol{z}}(n \delta) \delta+\frac{1}{2} \ddot{\boldsymbol{z}}(n \delta) \delta^2) / \gamma=-V(\boldsymbol{z}(n \delta))+\alpha J(\boldsymbol{z}(n \delta)) V(\boldsymbol{z}(n \delta))+\mathcal{O}\left(\alpha^2\right) .
\end{equation*}
Then, setting $\delta=\gamma$ and keeping the $\mathcal{O}(\alpha)$ terms we obtain:
\begin{equation*}
	\dot{\boldsymbol{z}}(t)+\frac{\gamma}{2} \ddot{\boldsymbol{z}}(t)=-V(\boldsymbol{z}(t))+\alpha \cdot J(\boldsymbol{z}(t)) \cdot V(\boldsymbol{z}(t)),
\end{equation*}
which yielding
\begin{equation*} 
	\begin{aligned}
		& \dot{\boldsymbol{z}}(t)=\boldsymbol{\omega}(t) \\
		& \dot{\boldsymbol{\omega}}(t)=-\beta \cdot 
		\boldsymbol{\omega}(t)-\beta \cdot V(\boldsymbol{z}(t))+ 
		\alpha \beta 
		\cdot J(\boldsymbol{z}(t)) \cdot V(\boldsymbol{z}(t)), \text 
		{ where } \beta=\frac{2}{\gamma} .
	\end{aligned} \tag{MPM-HRDE}
\end{equation*}

It is evident that in MPM-HRDE, if 
$\alpha = \gamma$, then MPM-HRDE will degenerate to EG-HRDE as mentioned in the literature \cite{chavdarova2021last, chavdarova2023last}. Therefore, this construction method of high-resolution ODE maintains the degenerate relationship of parameter values among discrete algorithms.

\section{Convergence of MPM-HRDE on bilinear game}
\indent
\par
In this section, we provide the proofs of the convergence results on 
the bilinear game (BG). As a reminder, we focus on the 
following problem with full rank $\boldsymbol{A} \in \mathbb{R}^{d_1} 
\times \mathbb{R}^{d_2}$
\begin{equation} \label{BG}
	\min _{\boldsymbol{x} \in \mathbb{R}^{d_1}} \max _{\boldsymbol{y} 
		\in \mathbb{R}^{d_2}} \boldsymbol{x}^{\boldsymbol{\top}} 
	\boldsymbol{A} \boldsymbol{y}
\end{equation}
The joint vector field of BG and its Jacobian are:
\begin{equation*} 
	V_{\mathrm{BG}}(\boldsymbol{z})=\left[\begin{array}{c}
		\boldsymbol{A} \boldsymbol{y} \\
		-\boldsymbol{A}^{\boldsymbol{\top}} \boldsymbol{x}
	\end{array}\right], \qquad
	J_{\mathrm{BG}}(\boldsymbol{z})=\left[\begin{array}{cc}
		0 & \boldsymbol{A} \\
		-\boldsymbol{A}^{\boldsymbol{\top}} & 0
	\end{array}\right].
\end{equation*}

By replacing the above in the derived HRDEs one can analyze the 
convergence of the modified predictive method on the bilinear game.
Moreover, as the obtained system is linear this can be done using 
standard tools from dynamical systems, without designing Lyapunov 
functions. Using the Routh Hurwitz criteria, the 
Theorems below state the convergence results for the modified 
predictive method. 

\begin{theorem}[Unique solution of MPM-HRDE] \label{theorem1}
MPM-HRDE has a unique
solution when applied to bilinear game \eqref{BG}. 
\end{theorem}

\begin{proof}

The differential equations for MPM take the form $\dot{\boldsymbol{u}}=\mathbf{G}(\boldsymbol{u})$, where $\boldsymbol{u}=(\boldsymbol{z}, \boldsymbol{\omega})$ and $\mathbf{G}$ is the vector field defined in (MPM-HRDE). If we can demonstrate that $\mathbf{G}$ is Lipschitz continuous, then we can apply Theorem 10 from \cite{shi2018understanding}, leading to Theorem \ref{theorem1}.
We may assume that $\boldsymbol{\omega}(t)$ and $\dot{\boldsymbol{\omega}}(t)$ are bounded, namely,
\begin{equation*}
\sup \limits_{0 \leq t \leq \infty}\|\boldsymbol{\omega}(t)\| \leq \mathcal{C}_1, \quad
\sup \limits_{0 \leq t \leq \infty}\|\dot{\boldsymbol{\omega}}(t)\| \leq \mathcal{C}_2.
\end{equation*}
Now, we will demonstrate the Lipschitz continuity of $\mathbf{G}$. Derived from MPM, we have:
\begin{equation*}
	\frac{d}{d t}\left[\begin{array}{l}
		\boldsymbol{z}_s \\
		\boldsymbol{\omega}_s
	\end{array}\right]=\left[\begin{array}{c}
		\boldsymbol{\omega}_s \\
		-\frac{2}{\gamma} \boldsymbol{\omega}_s-\frac{2}{\gamma} V\left(\boldsymbol{z}_s\right) + \frac{2\alpha}{\gamma} J\left(\boldsymbol{z}_s\right) V\left(\boldsymbol{z}_s\right)
	\end{array}\right].
\end{equation*}
Hence, for any $[\boldsymbol{z}_1, \boldsymbol{\omega}_1], [\boldsymbol{z}_2, \boldsymbol{\omega}_2]$  with the initial condition bounds on the norm of
$\boldsymbol{z}_1, \boldsymbol{z}_2, \boldsymbol{\omega}_1, \boldsymbol{\omega}_2$, we obtain:
\begin{equation*}
	\begin{aligned}
		& \left\|\left[\begin{array}{c}
			\boldsymbol{\omega}_1 \\
			-\frac{2}{\gamma} \boldsymbol{\omega}_1-\frac{2}{\gamma} V\left(\boldsymbol{z}_1\right)+\frac{2\alpha}{\gamma} J\left(\boldsymbol{z}_1\right) V\left(\boldsymbol{z}_1\right)
		\end{array}\right]-\left[\begin{array}{c}
			\boldsymbol{\omega}_2 \\
			-\frac{2}{\gamma} \boldsymbol{\omega}_2-\frac{2}{\gamma} V\left(\boldsymbol{z}_2\right)+ \frac{2\alpha}{\gamma} J\left(\boldsymbol{z}_2\right) V\left(\boldsymbol{z}_2\right)
		\end{array}\right]\right\|_2 \\
		& \leq\left\|\begin{array}{c}
			\boldsymbol{\omega}_1-\boldsymbol{\omega}_2 \\
			-\frac{2}{\gamma}\left(\boldsymbol{\omega}_1-\boldsymbol{\omega}_2\right)
		\end{array}\right\|_2+ 
	   \frac{2}{\gamma}\left\|\begin{array}{c}
		0 \\
		 \left(V\left(\boldsymbol{z}_1\right)- V\left(\boldsymbol{z}_2\right)\right)
	\end{array}\right\|_2   +  
   \frac{2\alpha}{\gamma}\left\|\begin{array}{c}
0 \\
J\left(\boldsymbol{z}_1\right) \left(V\left(\boldsymbol{z}_1\right)- V\left(\boldsymbol{z}_2\right)\right)
\end{array}\right\|_2\\
		& \quad+\frac{2\alpha}{\gamma}\left\|\begin{array}{c}
			0 \\
			V\left(\boldsymbol{z}_2\right) \left(J\left(\boldsymbol{z}_1\right)-J\left(\boldsymbol{z}_2\right)\right)
		\end{array}\right\|_2  \\
	&	\leq  \sqrt{1+\frac{4}{\gamma^2}}\left\|\boldsymbol{\omega}_1-\boldsymbol{\omega}_2\right\|_2+\frac{2}{\gamma} L_1\left\|\boldsymbol{z}_1-\boldsymbol{z}_2\right\|_2+ \frac{2\alpha}{\gamma} \underbrace{\left\|J\left(\boldsymbol{z}_1\right)\right\|_{\rm F}}_{\leq \sqrt{2}\left\|A\right\|_{\rm F}} \left\| V\left(\boldsymbol{z}_1\right)- V\left(\boldsymbol{z}_2\right)\right\|_2 \\
		& \quad +\frac{2\alpha}{\gamma} \underbrace{\left\|V\left(\boldsymbol{z}_2\right)\right\|_2}_{\leq \mathcal{C}_1} \underbrace{\left\|J\left(\boldsymbol{z}_1\right)-J\left(\boldsymbol{z}_2\right)\right\|_{\rm F}}_{=0} \\
		& \leq \left(\frac{2}{\gamma} L_1\right)\left\|\boldsymbol{z}_1-\boldsymbol{z}_2\right\|_2+\left(\sqrt{1+\frac{4}{\gamma^2}}+ \frac{2\sqrt{2} \alpha}{\gamma}\left\|A\right\|_{\rm F} \right)\left\|\boldsymbol{\omega}_1-\boldsymbol{\omega}_2\right\|_2 \\
		& \leq \sqrt{2} \max \left\{\frac{2}{\gamma} L_1, \sqrt{1+\frac{4}{\gamma^2}}+\frac{2\sqrt{2} \alpha}{\gamma}\left\|A\right\|_{\rm F}\right\}\left\|\left[\begin{array}{c}
			\boldsymbol{z}_1 \\
			\boldsymbol{\omega}_1
		\end{array}\right]-\left[\begin{array}{c}
			\boldsymbol{z}_2 \\
			\boldsymbol{\omega}_2
		\end{array}\right]\right\|_2,
	\end{aligned}
\end{equation*}
In conclusion, we can apply Theorem 10 from \cite{shi2018understanding}, leading to the conclusion of Theorem \ref{theorem1}. 
The proof is completed.
\end{proof}

\begin{remark}
In the proof of the uniqueness of HRDE solutions in binary non-cooperative games presented in Section B.3 of reference \cite{chavdarova2023last} and A.3 of reference \cite{chavdarova2021last}, the global Lipschitz constant may contain information about the norm of the variable $\boldsymbol{z}$. Since the binary linear game 
$J_{BG}(\boldsymbol{z})$ is constant, our scaling may be more reasonable.
\end{remark}

\begin{remark}
In Section B.3 of reference \cite{chavdarova2023last} and Section A.3 of reference \cite{chavdarova2021last}, when discussing the unique solution of the HRDEs for binary linear games, note that $J_{BG}(\boldsymbol{z}_1)$
is a constant value. It is difficult to achieve its norm being smaller than the Lipschitz constant multiplied by the norm of any 
$ \boldsymbol{z}_1$. Therefore, it is more reasonable to directly compute its Frobenius norm.
\end{remark}

\begin{theorem}\label{TH1}  For 
 $\alpha, \gamma >0$, 
the real part of the
eigenvalues of $\boldsymbol{C}_{MPM}$ is always negative, 
$\mathfrak{R}\left(\lambda_i\right)<0, \forall \lambda_i \in { S 
p}\left(\boldsymbol{C}_{MPM}\right)$, 
thus the $\mathcal{O}(\alpha)$-HRDE of the modified predictive method (MPM-HRDE) dynamics converges on the BG problem with $\alpha > 2 \gamma $.

\end{theorem}

\begin{proof}
Remember that for the MPM, we utilize the following equation (MPM-HRDE).
\begin{equation*} \label{MPM-HRDE}
	\begin{aligned}
		& \dot{\boldsymbol{z}}(t)=\boldsymbol{\omega}(t) \\
		& \dot{\boldsymbol{\omega}}(t)=-\beta \cdot 
		\boldsymbol{\omega}(t)-\beta \cdot V(\boldsymbol{z}(t))+ 
		\alpha \beta 
		\cdot J(\boldsymbol{z}(t)) \cdot V(\boldsymbol{z}(t)), \text 
		{ where } \beta=\frac{2}{\gamma} .
	\end{aligned}
\end{equation*}
If we denote $\dot{\boldsymbol{x}}(t)=\boldsymbol{\omega}_x(t)$ and 
$\dot{\boldsymbol{y}}(t)=\boldsymbol{\omega}_y(t)$ for MPM, then we have the following: 
\begin{equation*}
	\left[\begin{array}{c}
		\dot{\boldsymbol{x}}(t) \\
		\dot{\boldsymbol{y}}(t) \\
		\dot{\boldsymbol{\omega}}_x(t) \\
		\dot{\boldsymbol{\omega}}_y(t)
	\end{array}\right]=\underbrace{\left[\begin{array}{cccc}
			0 & 0 & \boldsymbol{I} & 0 \\
			0 & 0 & 0 & \boldsymbol{I} \\
			-\alpha \beta \boldsymbol{A} \boldsymbol{A}^{\top} & 
			-\beta 
			\boldsymbol{A} & -\beta \boldsymbol{I} & 0 \\
			\beta \boldsymbol{A}^{\top} & -\alpha \beta 
			\boldsymbol{A}^{\top} 
			\boldsymbol{A} & 0 & -\beta \boldsymbol{I}
		\end{array}\right]}_{\triangleq 
		\boldsymbol{C}_{\mathrm{MPM}}} 
		\cdot\left[\begin{array}{c}
		\boldsymbol{x}(t) \\
		\boldsymbol{y}(t) \\
		\boldsymbol{\omega}_x(t) \\
		\boldsymbol{\omega}_y(t)
	\end{array}\right] .
\end{equation*}
To determine the eigenvalues $\lambda \in \mathbb{C}$ of 
$\boldsymbol{C}_{\mathrm{MPM}}$, we do the following:
\begin{equation*}
	\begin{aligned}
		\operatorname{det}\left(\boldsymbol{C}_{\mathrm{MPM}}-\lambda 
		\boldsymbol{I}\right) & 
		=\operatorname{det}\left(\left[\begin{array}{cccc}
			-\lambda \boldsymbol{I} & 0 & \boldsymbol{I} & 0 \\
			0 & -\lambda \boldsymbol{I} & 0 & \boldsymbol{I} \\
			-\alpha \beta \boldsymbol{A} \boldsymbol{A}^{\top} & 
			-\beta 
			\boldsymbol{A} & -(\beta+\lambda) \boldsymbol{I} & 0 \\
			\beta \boldsymbol{A}^{\top} & -\alpha \beta 
			\boldsymbol{A}^{\top} 
			\boldsymbol{A} & 0 & -(\beta+\lambda) \boldsymbol{I}
		\end{array}\right]\right) \\
		& =\operatorname{det}\left(\lambda(\beta+\lambda) 
		\boldsymbol{I}-\underbrace{\left[\begin{array}{cc}
				-\alpha \beta \boldsymbol{A} \boldsymbol{A}^{\top} & 
				-\beta 
				\boldsymbol{A} \\
				\beta \boldsymbol{A}^{\top} & -\alpha \beta 
				\boldsymbol{A}^{\top} \boldsymbol{A}
			\end{array}\right]}_{\triangleq \boldsymbol{D}}\right),
	\end{aligned}
\end{equation*}
where we used $\operatorname{det}\left(\left[\begin{array}{ll}
	\boldsymbol{B}_1 & \boldsymbol{B}_2 \\
	\boldsymbol{B}_3 & \boldsymbol{B}_4
\end{array}\right]\right)=\operatorname{det}\left(\boldsymbol{B}_1 
\boldsymbol{B}_4-\boldsymbol{B}_2 \boldsymbol{B}_3\right)$.

Assume $\mu=\mu_1+\mu_2 i \in \mathbb{C}$ represents the eigenvalues of $\boldsymbol{D}$. Then, we have $\lambda(\beta+\lambda)-\mu=0$. By applying the generalized Hurwitz theorem for polynomials with complex coefficients \cite{xie1985stable}, we derive the following generalized Hurwitz array:
 
\begin{center}
\begin{tabular}{cccc}
	$\lambda^2$ &  \cellcolor[rgb]{.9,.9,.9}{
	1}  & 0 & $-\mu_1$ \\
    $\lambda^1$ & \cellcolor[rgb]{.9,.9,.9} {$\beta$} & $-\mu_2$ & 0 
    \\
    & $\mu_2$ & $-\beta \mu_1$ & 0 \\
    $\lambda^0$ & \cellcolor[rgb]{.9,.9,.9}{$-\mu_2^2-\beta^2 \mu_1$} 
    & 0 & 0.
\end{tabular}
\end{center}
In this array, the terms that change sign determine the stability of the polynomial. Notably, since $\beta>0$, the system is stable if and only if $\mu_1<-\frac{1}{\beta^2} 
\mu_2^2$.
Therefore, it is adequate to demonstrate that: 
\begin{equation} \label{eq44}
	\mathfrak{R}(\mu(\boldsymbol{z}))
	<-\frac{1}{\beta^2}\left(\mathfrak{I}(\mu(\boldsymbol{z}))\right)^2
	 .
\end{equation}

By ensuring that  the complex vector $\boldsymbol{z}$ satisfies the condition $\|\boldsymbol{z}\|_2=1 $, we can derive:
\begin{equation*}
	\begin{aligned}
		& \mu(\boldsymbol{z})=\overline{\boldsymbol{z}}^{\top} 
		\boldsymbol{D} \boldsymbol{z}=\left[\begin{array}{ll}
			\overline{\boldsymbol{x}}^{\top} & 
			\overline{\boldsymbol{y}}^{\top}
		\end{array}\right]\left[\begin{array}{cc}
			-\alpha \beta \boldsymbol{A} \boldsymbol{A}^{\top} & 
			-\beta 
			\boldsymbol{A} \\
			\beta \boldsymbol{A}^{\top} & -\alpha \beta 
			\boldsymbol{A}^{\top} 
			\boldsymbol{A}
		\end{array}\right]\left[\begin{array}{l}
			\boldsymbol{x} \\
			\boldsymbol{y}
		\end{array}\right] \\
		& =-\alpha \beta \|\boldsymbol{A}^{\top} 
		\boldsymbol{x} \|_2^2-\alpha \beta \|\boldsymbol{A} 
		\boldsymbol{y} 
		\|_2^2+\beta\left(\overline{\boldsymbol{y}}^{\boldsymbol{\top}}
		 \boldsymbol{A}^{\boldsymbol{\top}} 
		\boldsymbol{x}-\overline{\boldsymbol{x}}^{\boldsymbol{\top}} 
		\boldsymbol{A} \boldsymbol{y}\right) \\
		& =\underbrace{-\alpha 
		\beta \|\boldsymbol{A}^{\boldsymbol{\top}} 
		\boldsymbol{x} \|_2^2-\alpha \beta\|\boldsymbol{A} 
		\boldsymbol{y}\|_2^2}_{\mathfrak{R}(\mu(\boldsymbol{z}))}+\underbrace{2
		 \beta \cdot 
		\mathfrak{I}\left(\overline{\boldsymbol{x}}^{\boldsymbol{\top}}
		 \boldsymbol{A} 
		\boldsymbol{y}\right)}_{\mathfrak{I}(\mu(\boldsymbol{z}))} 
		\cdot i, 
	\end{aligned}
\end{equation*}
Considering the last equality, it stems from the fact that $\overline{\boldsymbol{x}}^{\top} \boldsymbol{A} \boldsymbol{y}$ is the complex conjugate of $\overline{\boldsymbol{y}}^{\top} 
\boldsymbol{A}^{\top} \boldsymbol{x}$. Thus, $\overline{\boldsymbol{y}}^{\top} 
\boldsymbol{A}^{\top} \boldsymbol{x}- 
\overline{\boldsymbol{x}}^{\top} \boldsymbol{A} \boldsymbol{y}=2 
\mathfrak{I}\left(\overline{\boldsymbol{x}}^{\top} \boldsymbol{A} 
\boldsymbol{y}\right) \cdot i$. Substituting this into \eqref{eq44}, we need to demonstrate:

\begin{equation*}
	-\alpha \beta\left( \|\boldsymbol{A}^{\boldsymbol{\top}} 
	\boldsymbol{x} \|_2^2+\|\boldsymbol{A} 
	\boldsymbol{y}\|_2^2\right) \leq-4 
	\mathfrak{I}^2\left(\overline{\boldsymbol{x}}^{\boldsymbol{\top}} 
	\boldsymbol{A} \boldsymbol{y}\right)
\end{equation*}
Thus, it suffices to show that:
\begin{equation*}
	\alpha \beta \left(\|\boldsymbol{A}^{\boldsymbol{\top}} 
	\boldsymbol{x} \|_2^2+\|\boldsymbol{A} \boldsymbol{y}\|_2^2 
	\right)
	\geq 4 
	\left| \left(\overline{\boldsymbol{x}}^{\boldsymbol{\top}} 
	\boldsymbol{A} \boldsymbol{y}\right) \right |^2
\end{equation*}
Hence, it is sufficient for the parameters $\alpha$ and 
$\beta$ to meet the following relationship.
\begin{equation*}
	\alpha > \frac{4}{\beta} \frac{ \left | 
	(\overline{\boldsymbol{x}}^{\boldsymbol{\top}} 
		\boldsymbol{A} \boldsymbol{y} ) \right |^2} 
		{ (\|\boldsymbol{A}^{\boldsymbol{\top}} 
		\boldsymbol{x} \|_2^2+\|\boldsymbol{A} \boldsymbol{y}\|_2^2 
		)} =2\gamma \frac{ \left|  
		(\overline{\boldsymbol{x}}^{\boldsymbol{\top}} 
		\boldsymbol{A} \boldsymbol{y} ) \right |^2} 
	{ (\|\boldsymbol{A}^{\boldsymbol{\top}} 
	\boldsymbol{x} \|_2^2+\|\boldsymbol{A} \boldsymbol{y}\|_2^2 
	)} 
\end{equation*}

On the other hand, due to
\begin{equation*}
\frac{ \left|   
	(\overline{\boldsymbol{x}}^{\boldsymbol{\top}} 
	\boldsymbol{A} \boldsymbol{y} ) \right |^2} 
{ (\|\boldsymbol{A}^{\boldsymbol{\top}} 
	\boldsymbol{x} \|_2^2+\|\boldsymbol{A} \boldsymbol{y}\|_2^2 
	)}  \leq
\frac{  \left \|   \overline{\boldsymbol{x}}^{\boldsymbol{\top}} \right \|^2 
	 \left \| \boldsymbol{A} \boldsymbol{y}   \right \|^2} 
{ (\|\boldsymbol{A}^{\boldsymbol{\top}} 
	\boldsymbol{x} \|_2^2+\|\boldsymbol{A} \boldsymbol{y}\|_2^2 
	)} \leq 
\frac{ \left \| \boldsymbol{A} \boldsymbol{y}   \right \|^2} 
{ (\|\boldsymbol{A}^{\boldsymbol{\top}} 
	\boldsymbol{x} \|_2^2+\|\boldsymbol{A} \boldsymbol{y}\|_2^2 
	)} \leq 1.
\end{equation*}
Therefore, to ensure that the conclusion of Theorem 3.4 holds, it is sufficient to require that the step size $\alpha$ and 
$\gamma$ satisfies the following conditions:
\begin{equation*} 
	\alpha > 2\gamma.  
\end{equation*}
This completes the proof.
\end{proof}

\begin{remark}
	Note that references \cite{chavdarova2021last, chavdarova2023last} require matrices $\boldsymbol{A}$ to satisfy conditions $\|\boldsymbol{A} \boldsymbol{y}\|_2=\|\boldsymbol{y}\|_2$ and $\left\|\boldsymbol{A}^{\boldsymbol{\top}} \boldsymbol{x}\right\|_2=\|\boldsymbol{x}\|_2$ when proving the convergence of EG in bilinear game \eqref{BG}. In fact, we can relax this requirement by introducing a two-time scale.
\end{remark}

\section*{Acknowledgement}
\indent
\par
This work was supported by the Major Program of National Natural Science Foundation
of China (Grant Nos. 11991020 and 11991024), the Team Project of Innovation Leading Talent in Chongqing
(Grant No. CQYC20210309536), NSFC-RGC (Hong Kong) Joint Research Program (Grant No. 12261160365),
and the Scientific and Technological Research Program of Chongqing Municipal Education Commission (Grant
No. KJQN202300528).
\bibliographystyle{unsrt}
\bibliography{reference}

\begin{thebibliography}{10}

\bibitem{jordan2018dynamical}
Jordan~Michael I.
\newblock Dynamical, symplectic and stochastic perspectives on gradient-based
  optimization.
\newblock In {\em Proceedings of the International Congress of Mathematicians:
  Rio de Janeiro 2018}, pages 523--549, 2018.

\bibitem{wibisono2016variational}
Wibisono Andre, Wilson~Ashia C, and Jordan~Michael I.
\newblock A variational perspective on accelerated methods in optimization.
\newblock {\em Proceedings of the National Academy of Sciences},
  113(47):E7351--E7358, 2016.

\bibitem{wilson2021lyapunov}
Wilson~Ashia C, Recht Ben, and Jordan~Michael I.
\newblock A lyapunov analysis of accelerated methods in optimization.
\newblock {\em Journal of Machine Learning Research}, 22(113):1--34, 2021.

\bibitem{su2016differential}
Su~Weijie, Boyd Stephen, and Candes~Emmanuel J.
\newblock A differential equation for modeling nesterov's accelerated gradient
  method: Theory and insights.
\newblock {\em Journal of Machine Learning Research}, 17(153):1--43, 2016.

\bibitem{shi2018understanding}
Shi Bin, Du~Simon S, Jordan~Michael I, and Su~Weijie J.
\newblock Understanding the acceleration phenomenon via high-resolution
  differential equations.
\newblock {\em ArXiv:1810.08907}, 2018.

\bibitem{shi2022understanding}
Shi Bin, Du~Simon S, Jordan~Michael I, and Su~Weijie J.
\newblock Understanding the acceleration phenomenon via high-resolution
  differential equations.
\newblock {\em Mathematical Programming}, 195:79--148, 2022.

\bibitem{shi2019acceleration}
Shi Bin, Du~Simon S, Su~Weijie, and Jordan~Michael I.
\newblock Acceleration via symplectic discretization of high-resolution
  differential equations.
\newblock In {\em Advances in Neural Information Processing Systems},
  volume~32, 2019.

\bibitem{yuan2023analysis}
Yuan Yaxiang and Zhang Yi.
\newblock Analysis accelerated mirror descent via high-resolution {ODE}s.
\newblock {\em ArXiv:2308.03242}, 2023.

\bibitem{lu2022sr}
Lu~Haihao.
\newblock An {$O(s^r)$}-resolution {ODE} framework for understanding
  discrete-time algorithms and applications to the linear convergence of
  minimax problems.
\newblock {\em Mathematical Programming}, 194(1):1061--1112, 2022.

\bibitem{chavdarova2021last}
Chavdarova Tatjana, Jordan~Michael I, and Zampetakis Manolis.
\newblock Last-iterate convergence of saddle point optimizers via
  high-resolution differential equations.
\newblock {\em ArXiv:2112.13826}, 2021.

\bibitem{chavdarova2022continuous}
Chavdarova Tatjana, Hsieh Yaping, and Jordan~Michael I.
\newblock Continuous-time analysis for variational inequalities: An overview
  and desiderata.
\newblock {\em ArXiv:2207.07105}, 2022.

\bibitem{chavdarova2023last}
Chavdarova Tatjana, Jordan~Michael I, and Zampetakis Manolis.
\newblock Last-iterate convergence of saddle point optimizers via
  high-resolution differential equations.
\newblock {\em Minimax Theory and its Applications}, 8:333--380, 2023.

\bibitem{goodfellow2014generative}
Ian Goodfellow, Jean Pouget-Abadie, Mehdi Mirza, Bing Xu, David Warde-Farley,
  Sherjil Ozair, Aaron Courville, and Yoshua Bengio.
\newblock Generative adversarial nets.
\newblock {\em Advances in neural information processing systems}, 27, 2014.

\bibitem{chavdarovalast2021OPT}
Chavdarova Tatjana, Jordan~Michael I, and Zampetakis Manolis.
\newblock Last-iterate convergence of saddle point optimizers via
  high-resolution differential equations.
\newblock In {\em NeurIPS OPT2021: 13th Annual Workshop on Optimization for
  Machine Learning}, 2021.

\bibitem{liang2019interaction}
Liang Tengyuan and Stokes James.
\newblock Interaction matters: A note on non-asymptotic local convergence of
  generative adversarial networks.
\newblock In {\em The 22nd International Conference on Artificial Intelligence
  and Statistics}, pages 907--915, 2019.

\bibitem{li2024}
Li~Keke, Yang Xinmin, and Zhang Ke.
\newblock Training {GANs} with predictive centripetal acceleration (in
  chinese).
\newblock {\em Sci Sin Math}, 54:671--698, 2024.

\bibitem{li2024Alleviating}
Li~Keke, Tang Liping, and Yang Xinmin.
\newblock Alleviating limit cycling in training {GANs} with an optimization
  technique.
\newblock {\em Sci Sin Math}, {doi.org/10.1007/s11425-023-2296-5}, 2024.

\bibitem{xie1985stable}
Xie Xukai.
\newblock Stable polynomials with complex coefficients.
\newblock pages 324--325, 1985.

\end{thebibliography}

\end{document}